\documentclass[11pt]{amsart} 
\usepackage{amssymb,amsmath,latexsym,enumerate,graphicx,color,multicol,bbm,mathptmx,microtype,auto-pst-pdf,tikz}

\newtheorem{theorem}{Theorem} 
\newtheorem{corollary}[theorem]{Corollary}
\newtheorem{proposition}[theorem]{Proposition}
\newtheorem{lemma}[theorem]{Lemma}

\newtheorem{exam}{Example}
\newenvironment{example}{\begin{exam}\rm}{\end{exam}}

\newtheorem*{rem}{Remarks}

\def\th{$^{\text{th}}$}
\def\mult{\operatorname{mult}}
\def\Z{\mathbb{Z}}

\def\R{\mathbb{R}}

\def\G{\mathcal{G}}
\def\HH{\mathcal{H}}
\def\P{\mathcal{P}}
\def\e{\mathbf{e}}
\def\m{\mathbf{m}}
\def\w{\mathbf{w}}
\def\x{\mathbf{x}}
\def\z{\mathbf{z}}
\begin{document}

\title{Enumeration of Golomb Rulers and Acyclic Orientations of Mixed Graphs}

\author{Matthias Beck}
\address{Department of Mathematics\\
         San Francisco State University\\
         1600 Holloway Ave\\
         San Francisco, CA 94132\\
         U.S.A.}
\email{mattbeck@sfsu.edu}

\author{Tristram Bogart}
\address{Departamento de Matem\'aticas \\
Universidad de los Andes \\
Cra 1 No. 18A-10, Edificio H \\
Bogot\'a, 111711 \\
Colombia}
\email{tc.bogart22@uniandes.edu.co}

\author{Tu Pham}
\address{Department of Mathematics\\
         University of California\\
         900 University Ave\\
         Riverside, CA 92521\\
         U.S.A.}
\email{pham@math.ucr.edu}

\dedicatory{Dedicated to our friend and mentor Joseph Gubeladze on the occation of his 50\th birthday}

\begin{abstract}
A \emph{Golomb ruler} is a sequence of distinct integers (the \emph{markings} of the ruler) whose pairwise differences are distinct.
Golomb rulers can be traced back to additive number theory in the 1930s and have attracted recent research activities on existence problems, such as the search for \emph{optimal} Golomb rulers (those of minimal length given a fixed number of markings).
Our goal is to enumerate Golomb rulers in a systematic way: we study
\[
  g_m(t) := \# \left\{ \x \in \Z^{m+1} : \, 0 = x_0 < x_1 < \dots < x_{ m-1 } < x_m = t , \text{ all } x_j - x_k \text{ distinct} \right\} ,
\]
the number of Golomb rulers with $m+1$ markings and length $t$.
Our main result is that $g_m(t)$ is a quasipolynomial in $t$ which satisfies a combinatorial reciprocity theorem:
$(-1)^{m-1} g_m(-t)$ equals the number of rulers $\x$ of length $t$ with $m+1$ markings, each counted with its \emph{Golomb multiplicity}, which measures how many combinatorially different Golomb rulers are in a small neighborhood of $\x$.
Our reciprocity theorem can be interpreted in terms of certain mixed graphs associated to Golomb rulers; in this language, it is reminiscent of Stanley's reciprocity theorem for chromatic polynomials.
Thus in the second part of the paper we develop an analogue of Stanley's theorem to mixed graphs, which connects their chromatic polynomials to acyclic orientations.
\end{abstract}

\date{25 October 2011}

\maketitle


\section{Introduction}

A \emph{Golomb ruler} is a sequence of $n$ distinct integers whose
pairwise differences are distinct: one can picture an actual ruler with
$n$ markings having the property that all possible measurements are of distinct length.
Golomb rulers have natural applications to phased array radio antennas (see, e.g., \cite{babcock}, possibly the earliest reference on Golomb rulers, though not under that name), x-ray analysis of crystal structures (see, e.g., \cite{bloomgolomb}), and error-correcting codes (see, e.g., \cite{klove}).
Mathematically they can be traced back to additive number theory in the 1930s \cite{erdosturan,sidon}.
In the more recent past, researchers have typically studied existence
problems, such as the search for \emph{optimal} Golomb rulers (those
of minimal length given a fixed number of markings), often with an eye
toward computational complexity. Figure \ref{introfigure}
shows an optimal ruler of length 6. See, e.g., \cite{alperindrobot,dollasrankinmccracken,drakakis,meyerpapakonstantinou,shearer} and the parallel-search project on Golomb rulers at {\tt http://www.distributed.net/Projects}.

\begin{figure}[h]
\begin{center}
\begin{tikzpicture}
\draw[line width=1mm](0,0)node[anchor=south]{0} -- (0,-.3) -- (6,-.3) --(6,0)node[anchor=south]{6};
\draw[line width=1mm](1,-.3)--(1,0)node[anchor=south]{1};
\draw[line width=1mm](4,-.3)--(4,0)node[anchor=south]{4};
\draw[|<->] (0,-.75) -- (1,-.75);
\node at (.5,-1){1};
\draw[<->] (1,-.75) -- (4,-.75);
\node at (2.5,-1){3};
\draw[<->|] (4,-.75) -- (6,-.75);
\node at (5,-1){2};
\draw[|<->|] (0,-1.5) -- (4,-1.5);
\node at (2,-1.75){4};
\draw[|<->|] (1,-2) -- (6,-2);
\node at (3.5,-2.25){5};
\end{tikzpicture}
\end{center}
\caption{An optimal Golomb ruler with four markings.}\label{introfigure}
\end{figure}
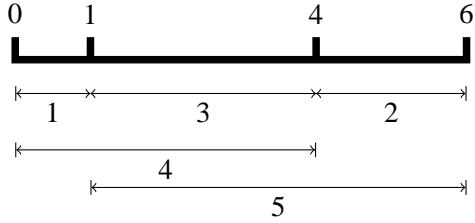

Our goal is to enumerate Golomb rulers in a systematic way. We define
\[
  g_m(t) := \# \left\{ \x \in \Z^{m+1} : \, 0 = x_0 < x_1 < \dots < x_{ m-1 } < x_m = t , \text{ all } x_j - x_k \text{ distinct} \right\} ,
\]
the number of Golomb rulers with $m+1$ markings and length $t$.
Most often it will be more convenient for us to express this counting function in the equivalent form
\begin{equation}\label{measurementrep}
  g_m(t) = \# \left\{ \z \in \Z_{ >0 }^m : \, \begin{array}{l}
    z_1 + z_2 + \dots + z_m = t , \\
    \sum_{ j \in U } z_j \ne \sum_{ j \in V } z_j \ \text{ for all dpcs } U, V \subset [m]
  \end{array} \right\}
\end{equation}
where $[m] := \left\{ 1, 2, \dots, m \right\}$, and dpcs stands for
\emph{disjoint proper consecutive subsets}, i.e., two sets of the form
$\left\{ a, a+1, a+2, \dots, b \right\}$, $\left\{ c, c+1, c+2, \dots, d \right\}$ for some $1 \le a \le b < c \le d \le m$.  

Motivated by \eqref{measurementrep}, we call a point $\z \in \Z_{ \ge
  0 }^m$ with $z_1 + z_2 + \dots + z_m = t$ a \emph{ruler} of
\emph{length} $t$; thus a Golomb ruler is a ruler with positive
entries satisfying $\sum_{ j \in U } z_j \ne \sum_{ j \in V } z_j$ for
all dpcs $U, V \subset [m]$. 

We define a \emph{real Golomb ruler} (with $m+1$ markings and length $t$) as a vector $\z \in \R_{ \ge 0 }^m$ satisfying
\[
  z_1 + z_2 + \dots + z_m = t
  \qquad \text{ and } \qquad
  \sum_{ j \in U } z_j \ne \sum_{ j \in V } z_j \ \text{ for all dpcs } \ U, V \subset [m] \, .
\]
Two real Golomb rulers $\z, \w \in \R_{ \ge 0 }^m$ are \emph{combinatorially equivalent} if for any dpcs $U, V \subset [m]$,
\[
  \sum_{ j \in U } z_j < \sum_{ j \in V } z_j
  \qquad \Longleftrightarrow \qquad
  \sum_{ j \in U } w_j < \sum_{ j \in V } w_j \, ,
\]
in plain English: if their possible measurements satisfy the same order relations.
We define the \emph{Golomb multiplicity} of a ruler $\z \in \Z_{ \ge 0 }^m$ to be the number of combinatorially different real Golomb rulers in an $\epsilon$-neighborhood of $\z$ (viewed as a point in $\R^m$), for sufficiently small $\epsilon > 0$. Thus a Golomb ruler is a ruler with Golomb multiplicity~1.

A \emph{quasipolynomial} of degree $d$ is a function of the form
$c_d(t) \, t^d + c_{ d-1 } (t) \, t^{ d-1 } + \dots + c_0(t)$,
where the $c_j(t)$ are periodic functions in $t$.  The lcm of their
periods is called the \emph{period} of this quasipolynomial.
Our main result about $g_m(t)$ is the following combinatorial
reciprocity theorem. 

\begin{theorem}\label{mainthm}
The Golomb counting function $g_m(t)$ is a quasipolynomial in $t$ of
degree $m-1$ with leading coefficient $\frac{1}{(m-1)!}$. Its evaluation $(-1)^{m-1} g_m(-t)$ equals the number of rulers in $\Z_{ \ge 0 }^m$ of length $t$, each counted with its Golomb multiplicity. Furthermore, $(-1)^{m-1} g_m(0)$ equals the number of combinatorially different Golomb rulers with $m+1$ markings.
\end{theorem}

\begin{example} \label{ex:g_3} 
Let $m=3$. Using James B.\ Shearer's Fortran code for Golomb rulers\footnote{Available at {\tt http://www.research.ibm.com/people/s/shearer/programs/grs1.txt}},
we computed the values of $g_3(t)$ for $6 \leq t \leq 35$. Note that
if $\z$ is a Golomb ruler with $m+1$ markings, then there is a
complementary Golomb ruler
$\z'$ given by $z_i' = z_{m-i}$, and $(\z')' = \z$. The rulers $\z$
and $\z'$ are always distinct unless $m=1$. Shearer's code
computes only one of each complementary pair, but in Table \ref{golextable}
we double the output values to account for both. We also note that by
inspection, $g_3(t) = 0$ for $t < 6$. 

\begin{table}[htb]
\begin{tabular}{|c|c||c|c||c|c|} \hline
$t$ & $g_3(t)$ & $t$ & $g_3(t)$ & $t$ & $g_3(t)$ \\ \hline
6 & 2 & 16 & 72 & 26 & 240 \\
7 & 6 & 17 & 96 & 27 & 288 \\
8 & 8 & 18 & 98 & 28 & 288 \\
9 & 18 & 19 & 126 & 29 & 336 \\ 
10 & 16 & 20 & 128 & 30 & 338 \\  
11 & 30 & 21 & 162 & 31 & 390 \\
12 & 34 & 22 & 160 & 32 & 392 \\
13 & 48 & 23 & 198 & 33 & 450 \\
14 & 48 & 24 & 202 & 34 & 448 \\
15 & 72 & 25 & 240 & 35 & 510 \\ \hline
\end{tabular}\caption{The Golomb counting function $g_3(t)$ for $6 \leq t \leq 35$.}\label{golextable}
\end{table}

\noindent
In Section \ref{iosection} we show that the period of the quasipolynomial function
$g_3(t)$ divides 12 and that its leading term is $\frac{1}{2}t^2$. Thus from
the 24 values $g_3(0),g_3(1),\dots,g_3(23)$ we can obtain all of the
coefficients by interpolation. The remaining 12 values
$g_3(24),\dots,g_3(35)$ are consistent with the result of this computation, which is: 

$$g_3(t) = 
\begin{cases}
\frac{1}{2} t^2 - 4t + 10 & \mbox{if } t \equiv 0, \\
\frac{1}{2} t^2 - 3t + \frac{5}{2} & \mbox{if } t \equiv 1, 5, 7, 11, \\
\frac{1}{2} t^2 - 4t + 6 & \mbox{if } t \equiv 2, 10, \\
\frac{1}{2} t^2 - 3t + \frac{9}{2} & \mbox{if } t \equiv 3, 9, \\
\frac{1}{2} t^2 - 4t + 8 & \mbox{if } t \equiv 4, 6, 8.
\end{cases} \pmod{12}$$
In particular, the coefficient $c_1(t)$ has period 2.
The coefficient $c_0(t)$ has period 12 but obeys the same
formula for $t \equiv j \pmod{12}$ as for $t \equiv -j \pmod{12}$ for
each $j$.
Note also that $g_3(0) = 10$; Theorem \ref{mainthm} predicts that there are ten combinatorially different Golomb rulers, as we will see in Section~\ref{iosection}.
\end{example}

We prove Theorem \ref{mainthm} geometrically, making use of the
machinery of \emph{inside-out polytopes} \cite{iop}.
This approach leads us to associate a mixed graph (i.e., a graph that may contain both undirected and directed edges; see, e.g., \cite{hararypalmer})
to the set of all Golomb rulers with a fixed number of markings.
In this language, Theorem \ref{mainthm} has an interpretation in terms
of acyclic orientations of mixed graphs, which might be of
independent interest (Theorem \ref{regionsthm} and Corollary
\ref{acycliccor}, proved in Section
\ref{regionsection}).
This reinterpretation is reminiscent of Stanley's reciprocity theorem for the chromatic polynomial of a graph \cite{stanleyacyclic}. This leads us to a natural analogue of Stanley's theorem for (general) mixed graphs, as follows.

We write a mixed graph $G$ as $G=(V,E,A)$ where $E$
represents the undirected edges and $A$ the directed edges.\footnote{
We assume that $G$ is \emph{simple}: that is, that it contains
neither multiple edges nor loops.}
A \emph{$t$-coloring} of a mixed graph $G$ is a map $c: V \to [t]$. Such a $t$-coloring is \emph{proper} if
\begin{itemize}
\item $c(v) \neq c(w)$ for all $\{v,w\} \in E$, and 
\item $c(v) < c(w)$ for all $(v,w) \in A$. 
\end{itemize} 

As with undirected graphs, the \emph{chromatic
  number} is the minimum $t$ such that $G$ can be
$t$-colored. This parameter was introduced for mixed graphs by Hansen,
Kuplinsky, and de Werra~\cite{HaKuWe}, who discuss bounds, algorithms,
and applications to scheduling problems. 
Let $\chi_G(t)$ denote the number of proper $t$-colorings of $G$.
Sotskov, Tanaev, and Werner~\cite{SoTaWe} showed that this function
(if not identically zero) is a polynomial of degree $|V|$ and computed the two leading coefficients.

An \emph{orientation} of a mixed graph is obtained by keeping each directed edge and orienting each undirected edge; an orientation is \emph{acyclic} if it does not contain any coherently-oriented cycles.
A coloring $c$ and an orientation of $G$ are \emph{compatible} if
\[
  c(u) \le c(v)
  \qquad \Longleftrightarrow \qquad
  u \to v.
\]
Note that $c$ may not be a proper coloring.

Here is our generalization of Stanley's chromatic reciprocity theorem to mixed graphs.  Our approach (in Section
\ref{graphsection}) also
yields a new proof of polynomiality of $\chi_G(t)$.

\begin{theorem} \label{thm:mixedgraph}
Let $G$ be a mixed graph. For a non-negative integer $t$,
$(-1)^{ |V| } \chi_G(-t)$ equals the number of $t$-colorings of $G$, each counted with multiplicity equal to the number of compatible acyclic orientations of $G$.
\end{theorem}


\begin{corollary}
For $G$ a mixed graph, $(-1)^{ |V| } \chi_G(-1)$ equals the number
of acyclic orientations of $G$.
\end{corollary}

In Section \ref{open} we end with
some open questions about Golomb rulers suggested by our geometric
approach. 


\section{Inside-out Polytopes}\label{iosection}

It is a short step to interpret a Golomb ruler in its measurement representation (as in \eqref{measurementrep}) as an integer lattice point confined to the positive orthant in $\R^m$ and the affine space defined through
\[
  z_1 + z_2 + \dots + z_m = t \, .
\]
More precisely, a Golomb ruler with $m+1$ markings and length $t$ is a lattice point in the $t$'th dilate of
\[
  \Delta_m^\circ := \left\{ \z \in \R_{ >0 }^m :  \, z_1 + z_2 + \dots + z_m = 1 \right\} ,
\]
the $(m-1)$-dimensional open \emph{standard simplex} living in $\R^m$.
However, this description does not yet capture the distinctness
condition: a Golomb-ruler lattice point must also avoid the hyperplanes in $\R^m$ given by equations of the form
\[
  \sum_{ j \in U } z_j = \sum_{ j \in V } z_j
\]
for all dpcs $U, V \subset [m]$.
Let $\G_m$ be the collection of all such hyperplanes.
Thus we have built a geometric setting in which to compute the Golomb counting function as
\begin{equation}\label{golombtoehrart}
  g_m(t) = \# \left( t \Delta_m^\circ \setminus \G_m \cap \Z^m \right) .
\end{equation}
In the language of \cite{iop}, $\left( \Delta_m, \G_m \right)$ is an \emph{inside-out polytope}\footnote{To be 100\% precise, \cite{iop} would compute $g_m(t)$ as $\# \left( \Delta_m^\circ \setminus \G_m \cap \frac 1 t \Z^m \right)$, but this is equivalent as the hyperplanes in $\G_m$ are linear.}
and $g_m(t)$ is the \emph{(open) Ehrhart quasipolynomial} of $\left( \Delta_m, \G_m \right)$.
Figure \ref{iopfigu} shows this geometric setting (viewed in the plane $z_1 + z_2 + z_3 = 1$) for the case $m=3$.

\begin{figure}[h!]
  \begin{center}
\begin{tikzpicture}[scale=.8]
[domain=0:2]
\draw[very thick, fill=gray!20] (90:3cm) -- (210:3cm) -- (330:3cm) -- (90:3cm);
\draw[-,very thick, color=gray](90:3cm) --(270:3cm)node[below left] {$z_1=z_2$};
\draw[-,very thick, color=gray](210:4cm) --(30:3cm)node[above right] {$z_2=z_3$};
\draw[-,very thick, color=gray](330:4cm) --(150:3cm)node[above left] {$z_1=z_3$};
\draw[-,very thick, color=gray](-3,.75) --(3,.75)node[below right] {$z_1 + z_2=z_3$};
\draw[-,very thick, color=gray](134.04:3.09cm) --(285.96:3.09cm)node[below right] {$z_1 = z_2+z_3$};
\end{tikzpicture}
\end{center}
\caption{The inside-out polytope $\left( \Delta_3, \G_3 \right)$.}\label{iopfigu}
\end{figure}
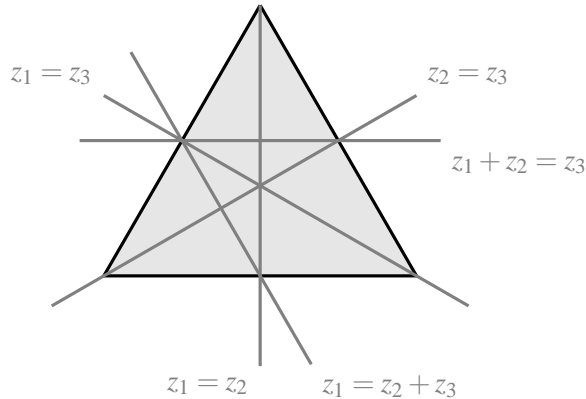

\noindent
The vertices of this inside-out polytope 
are 
$$\left(0,0,1\right),\left(\tfrac{1}{2},0,\tfrac{1}{2}\right),\left(\tfrac{1}{4},\tfrac{1}{4},\tfrac{1}{2}\right),\left(0,\tfrac{1}{2},\tfrac{1}{2}\right),\left(\tfrac{1}{3},\tfrac{1}{3},\tfrac{1}{3}\right),\left(\tfrac{1}{2},\tfrac{1}{4},\tfrac{1}{4}\right),\left(1,0,0\right),\left(\tfrac{1}{2},\tfrac{1}{2},0\right),\left(0,1,0\right) .$$
The lcm of the denominators among all of these coordinates is
twelve. This shows that the period of the quasipolynomial $g_3(t)$
divides twelve, as claimed in Example~\ref{ex:g_3}.   
We can also see that $\G_3$ dissects $\Delta_3$ into ten polygons, which correspond to the combinatorially different Golomb rulers.

The general setup of an inside-out polytope $(\P, \HH)$ consists of a
rational polytope $\P$ and a rational hyperplane arrangement $\HH$ in
$\R^d$; that is, the linear equations and inequalities defining $\P$ and $\HH$ have integer coefficients.
The goal is to compute the counting function
\[
  L_{ \P, \HH }^\circ (t) := \# \left( (\P \setminus \HH) \cap \tfrac 1 t \Z^d \right) ,
\]
and it follows from Ehrhart's theory of counting lattice points in
dilates of rational polytopes \cite{ccd,ehrhartpolynomial} that this
function is a quasipolynomial in $t$ whose degree is $\dim(P)$, whose
(constant) leading coefficient is the normalized lattice volume of $P$,
and whose period divides the lcm of all denominators that appear in
the coordinates of the vertices of $(\P, \HH)$. 
Furthermore, \cite{iop} established the reciprocity theorem
\begin{equation}\label{iopreciprocity}
  L_{ \P^\circ, \HH }^\circ (-t) = (-1)^{ \dim \P } L_{ \overline \P, \HH } (t) \, ,
\end{equation}
where $\P^\circ$ and $\overline \P$ denote the interior and closure of $\P$, respectively, and
\begin{equation}\label{multdef}
  L_{ \P, \HH } (t) := \sum_{ \m \in \frac 1 t \Z^d } \mult_{ \P, \HH } (\m)
\end{equation}
where $\mult_{ \P, \HH } (\m)$ denotes the number of closed regions of $(\P, \HH)$ containing $\m$.
(A \emph{region} of $(\P, \HH)$ is a connected component of $\P \setminus \HH$; a \emph{closed region} is the closure of a region.)
It follows from Ehrhart's work and \eqref{multdef} that
\begin{equation}\label{iopconst}
  L_{ \P, \HH } (0) = \# \text{ regions of } (\P, \HH).
\end{equation}
See \cite{iop} for this and several more properties of inside-out
polytopes. The concept of inside-out polytopes has been applied to a number of combinatorial settings; at the heart of any such application is an interpretation of the regions of $(\P, \HH)$, which we will now give for the Golomb inside-out polytope~$\left( \Delta_m, \G_m \right)$.

\begin{proof}[Proof of Theorem \ref{mainthm}]
The first statement follows immediately from the fact that $\left(
  \Delta_m, \G_m \right)$ is a rational inside-out polytope and 
$\Delta_m$ is a unimodular $(m-1)$-dimensional simplex. 

Viewing \eqref{iopreciprocity} from \eqref{golombtoehrart}, we know
that $(-1)^{m-1} g_m(-t)$ equals the number of rulers in $\Delta \cap
\frac 1 t \Z^m$, each counted with multiplicity given by the number of
closed regions of $\left( \Delta_m, \G_m \right)$ it lies in.
These regions, in turn, are defined by inequalities of the form
\[
  \sum_{ j \in U } z_j < \sum_{ j \in V } z_j
\]
for some dpcs $U, V \subset [m]$, and thus the multiplicity of a ruler $\z$ is precisely given by the number of combinatorially different real Golomb rulers in a neighborhood of~$\z$.

The last statement of Theorem \ref{mainthm} follows from~\eqref{iopconst}.
\end{proof}


\section{The Regions of $\left( \Delta_m, \G_m \right)$}\label{regionsection}

For each positive integer $m$, we define the \emph{Golomb graph}
$\Gamma_m$ to be a mixed graph whose vertices are all proper
consecutive subsets of $[m]$. The underlying graph is complete and an edge $UV$
is directed ($U \to V$) if and only if $U \subset V$. All other edges
are undirected. Acyclic orientations of a Golomb graph allow us to give an interpretation of the regions of $\left( \Delta_m, \G_m \right)$ in the following sense.

\begin{theorem}\label{regionsthm}
The regions of the Golomb inside-out polytope $\left( \Delta_m, \G_m
\right)$ are in one-to-one correspondence with the acyclic
orientations of the Golomb graph $\Gamma_m$ that satisfy 
the relation 
\begin{equation}\label{acycliccondition}
  A \to B
  \qquad \Longleftrightarrow \qquad
  U \to V \, .
\end{equation}
for all proper consecutive subsets $A$ and $B$ of $[m]$ of the form $A
= U \cup W$ and $B = V \cup W$ for some nonempty disjoint sets $U, V, W$
\end{theorem} 


\begin{figure}[h]
\begin{tikzpicture}[scale=1.5]
\tikzstyle{every node}=[draw,shape=circle];
\path (0,0) node (v1) {$1$};
\path (2,0) node (v2) {$2$};
\path (4,0) node (v3) {$3$};
\path (1,1) node (v4) {$12$};
\path (3,1) node (v5) {$23$};
\draw[line width=1pt] (v1) -- (v2)
(v1) .. controls (2,-1) .. (v3)
(v1) -- (v5)
(v2) -- (v3)
(v3) -- (v4)
(v4) -- (v5);

\draw[->,line width=1pt] (v1) -- (v4);
\draw[->,line width=1pt] (v2) -- (v4);
\draw[->,line width=1pt] (v2) -- (v5);
\draw[->,line width=1pt] (v3) -- (v5);

\path (5,0) node (w1) {$1$};
\path (7,0) node (w2) {$2$};
\path (9,0) node (w3) {$3$};
\path (6,1) node (w4) {$12$};
\path (8,1) node (w5) {$23$};
\draw[->,line width=1pt] (w1) -- (w2);
\draw[->,line width=1pt](w1) .. controls (7,-1) .. (w3);
\draw[->,line width=1pt](w1) -- (w5);
\draw[->,line width=1pt](w2) -- (w3);
\draw[->,line width=1pt](w3) -- (w4);
\draw[->,line width=1pt](w4) -- (w5);

\draw[->,line width=1pt] (w1) -- (w4);
\draw[->,line width=1pt] (w2) -- (w4);
\draw[->,line width=1pt] (w2) -- (w5);
\draw[->,line width=1pt] (w3) -- (w5);

\end{tikzpicture}
  \caption{The mixed graph $\Gamma_3$ (left) and an acyclic
    orientation of $\Gamma_3$ (right).}
  \label{G_3andAO}
\end{figure}
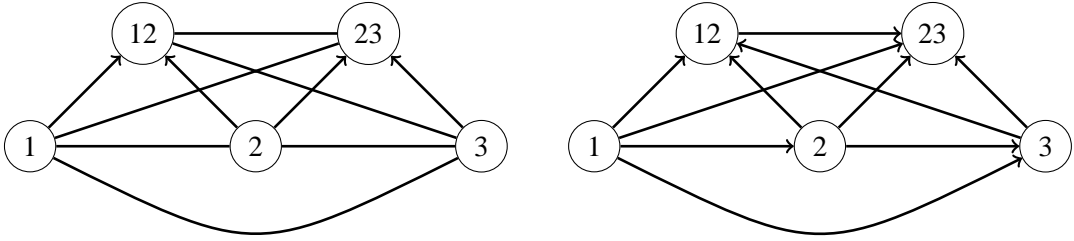 

\begin{proof}
Let $R$ be a region of $\left( \Delta_m, \G_m \right)$.
Thus $R$ is the intersection of $\Delta_m$ with halfspaces defined by inequalities of the form
\begin{equation}\label{halfspace}
  \sum_{ j \in U } z_j < \sum_{ j \in V } z_j
\end{equation}
for some dpcs $U, V \subset [m]$.
So $UV$ is an undirected edge of $\Gamma_m$; assign it the orientation $U \to V$.

The directed edges of $\Gamma_m$ are of the form $UV$ for $U \subset V$, so we still need to orient edges of the form $AB$ where $A = U \cup W$ and $B = V \cup W$ for some nonempty disjoint sets $U, V, W$. We orient these according to \eqref{acycliccondition} (note that the corresponding halfspace defined through $\sum_{ j \in A } z_j < \sum_{ j \in B } z_j$ is identical to that given by \eqref{halfspace}).
The resulting orientation of $\Gamma_m$ is acyclic since a coherently-oriented cycle gives rise to the nonsensical inequality $\sum_{ j \in M } z_j < \sum_{ j \in M } z_j$ for some multiset~$M$.

Conversely, suppose we are given an acyclic orientation of $\Gamma_m$ that satisfies \eqref{acycliccondition}.
Index the variables $y_U$ of a vector in $\R^{ m(m+1)/2 - 1 }$ by the proper consecutive subsets $U$ of $[m]$ and consider the \emph{braid arrangement} given by the hyperplanes $y_U = y_V$ in $\R^{ m(m+1)/2 - 1 }$. 
It is a famous observation of Greene \cite{greeneacyclic} that the acyclic orientations of a graph are in one-to-one correspondence with the region of the corresponding graphical arrangement (in our case, the above braid arrangement, since our underlying graph is complete).
Now consider the function $\R^{ m(m+1)/2 - 1 } \to \R^m$ defined through
\[
  y_U \mapsto \sum_{ j \in U } z_j \, .
\]
This is a linear map of full rank, i.e., the regions of the braid
arrangement, being full-dimensional sets in the domain, are mapped to full-dimensional sets in the range.
Furthermore, since the all-ones vector is in each closed region, we can project further to the affine space defined by $z_1 + z_2 + \dots + z_m = 1$.
\end{proof}

This theorem allows us to rephrase Theorem \ref{mainthm} in terms of acyclic orientations of $\Gamma_m$.
We call a ruler $\z \in \Z_{ \ge 0 }^m$ and an orientation of $\Gamma_m$ \emph{compatible} if
\[
  \sum_{ j \in U } z_j \le \sum_{ j \in V } z_j
  \qquad \Longleftrightarrow \qquad
  U \to V \, .
\]

\begin{corollary}\label{acycliccor}
The evaluation $(-1)^{m-1} g_m(-t)$ equals the number of rulers in $\Z_{ \ge 0 }^m$ of length $t$, each counted with multiplicity equal to the number of compatible acyclic orientations of $\Gamma_m$ that satisfy \eqref{acycliccondition}. Furthermore, $(-1)^{m-1} g_m(0)$, the number of combinatorially different Golomb rulers, equals the number of acyclic orientations of $\Gamma_m$ that satisfy~\eqref{acycliccondition}.
\end{corollary}


\section{Mixed Graph Coloring Reciprocity}\label{graphsection}

Our proof of Theorem \ref{thm:mixedgraph} closely follows the proof of Stanley's chromatic reciprocity theorem given in \cite{iop}.
Given a mixed graph $G=(V,E,A)$ with $V=[n]$, define the polytope 
\[
  \P(G) := \left\{ \x \in [0,1]^n: \, x_j \le x_k \text{ whenever } (j,k) \in A \right\}
\]
and the hyperplane arrangement $\HH(G)$ consisting of the hyperplanes $x_j
= x_k$ for all $\{v_j,v_k\} \in E$. 
The following result is the mixed-graph analogue of \cite{greeneacyclic}.

\begin{proposition} \label{prop:regions} The regions of $\R^n \setminus \HH(G)$ that
  intersect $\P(G)$ are in bijection with the acyclic
  orientations of~$G$.
\end{proposition}

\begin{proof}
 Let $R$ be an open region of $\HH(G)$ that intersects $\P(G)$, hence
 also intersects $\P(G)^\circ$. Then for each hyperplane
$x_i = x_j$ in $\HH(G)$, $R$ is entirely on one side, hence determines
an orientation $\alpha$ as follows: for each edge $\{ i, j \} \in E$, choose $i \to j$ if $x_i < x_j$ in $R$, or $j \to i$ if $x_i > x_j$
in $R$.  In the set $R \cap \P(G)^\circ$, we also have $x_i < x_j$ for
all $(i,j) \in A$. Since $R$, being open, intersects the interior of
$\P(G)^\circ$, all of these inequalities must be consistent. That is, $\alpha$ is acylic. 

Since two different regions must be on opposite sides of at
least one hyperplane, they determine different orientations. Finally,
if $\alpha$ is an ayclic orientation of $G$, choose any total order $\pi$
of the vertices that is consistent with $\alpha$, including the fixed
directions on the arrows $A$. Then if $R$ is the region containing the points of
$\R^n$ that satisfy $x_{\pi_1} < \dots < x_{\pi_n}$, then $R$ maps to
$\alpha$ under the rule given above. That is, the map from regions to
orientations is both injective and surjective. 
\end{proof}

We say that $G$ is \emph{acylic} if the directed graph $(G,A)$ contains no directed cycles. Note that in this case there exists at least one total order on the vertices compatible with the directed edges. Orienting the undirected edges according to such an order, we can always obtain an acyclic orientation on all of $G$.   
Note that if $G$ is not acyclic, then there exist no proper
$t$-colorings for any $t$. On the other hand, if $G$ is acyclic then
there exist proper $t$-colorings for all sufficiently large $t$, and in particular for all $t \geq n$.  

\begin{lemma}
If the mixed graph $G$ is acyclic then the polytope $\P(G)$ is full dimensional. 
\end{lemma}

\begin{proof}
The polytope $\P(G)$ depends only on the directed
graph $D=(V,A)$. Since $D$ is acyclic, there is a total order $i_1 \prec
i_2 \prec \dots \prec i_n$ on $[n]$ such that $j \prec k$ for all
$(v_j,v_k) \in A$. Thus $$\P(G) \supseteq \{ \x \in  \R^n: \, 0 \leq
x_{i_1} \leq x_{i_2} \leq \dots \leq x_{i_n} \leq 1 \}$$
which is a full-dimensional simplex in $\R^n$, so $\P(G)$ is
full-dimensional. 
\end{proof}

\begin{lemma} \label{lem:integral}For any acyclic mixed graph $G$, the
  inside-out polytope $(\P(G), \HH(G))$ is integral. 
\end{lemma}

\begin{proof} Any closed region of $(\P(G), \HH(G))$ is defined by a
  collection of inequalities of the form $x_j - x_k \leq 0$ and $x_j
  \geq 0$. In particular, the left-hand sides of these inequalities
  are all vectors of the form $\e_j - \e_k$ or $-\e_j$. Any matrix whose
  rows are all such vectors is totally unimodular. This implies that all
  vertices of the region are integral~\cite{schrijver}. 
\end{proof}

\begin{proof}[Proof of Theorem~\ref{thm:mixedgraph}]
Set $\P:=\P(G)$ and $\HH:=\HH(G)$ . 
Using the identification of maps $c: V \to \Z$ with lattice points
in $\R^n$, we see that for $t \geq 1$, $c$ is a $(t-1)$-coloring (respecting the
directed edges $A$) if $c \in t\P^\circ(G)$ and $c \notin \HH$. That is,
$\chi_G(t-1) = E^\circ_{\P^\circ,\HH}(t)$, which equals $(-1)^n E_{\P,\HH}(-t)$ by \eqref{iopreciprocity}.
Using Lemma~\ref{lem:integral}, we see that the functions $E_{\P,\HH}$ and
$E^\circ_{\P^\circ,\HH}$ are polynomials and hence the reciprocity extends to
all values of $t$. In particular, for $s$ a positive integer we have 
$\chi_G(-s-1) = E^\circ_{\P^\circ,\HH}(-s) = (-1)^n E_{\P,\HH}(s)$. Now 
$E_{\P,\HH}(s)$ counts each point $c \in s\P$ with multiplicity equal to
the number of closed regions of $\HH$ that intersect the interior of
$\P$ and contain $c$. By Proposition~\ref{prop:regions}, regions
correspond to acyclic orientations of $G$. The closed region
corresponding to an orientation $\alpha$ contains $c$ if and only if $c(u)
\leq c(v)$ for each edge $u \to v$ in $\alpha$. 
\end{proof}

\begin{example} Let $G = \left([3], \{ \{1,3\}, \{2,3\} \}, \{ (1,2)
    \}\right)$, a triangle with one directed edge and two undirected
  edges. Of the four orientations of $G$, three are acyclic and each
  of these three determines a total order $\prec$ on the
  vertices. These three orders are: $$ 1 \prec 2 \prec 3,\, 1 \prec 3 \prec 2,\, 3 \prec 1 \prec 2 \, .$$ 


\noindent
The polytope $\P(G)$ is the half-cube given by the five inequalities 
$$ 0 \leq x_1 \leq x_2 \leq 1, \, 0 \leq x_3 \leq 1 \, .$$
The arrangement $\HH(G)$ consists of the two hyperplanes $x_1 = x_3$ and
$x_2 = x_3$, and the inside-out polytope $(\P(G), \HH(G))$ has three open
regions in bijection with the acyclic orientations of $G$:
\begin{align*}
R_{123} &= \left\{\x \in \R^3: 0 < x_1 < x_2 < x_3 < 1 \right\} \\
R_{132} &= \left\{\x \in \R^3: 0 < x_1 < x_3 < x_2 < 1 \right\} \\
R_{312} &= \left\{\x \in \R^3: 0 < x_3 < x_1 < x_2 < 1 \right\} .
\end{align*}

For this graph it is easy to compute $\chi_G(t)$ directly. In any coloring $c$, we must
choose $c(1) < c(2)$ and then we can choose $c(3)$ to be any remaining
color. Thus there are $ {t \choose 2} (t-2) = \frac{t(t-1)(t-2)}{2}$
$t$-colorings for any non-negative integer $t$. Then $(-1)^3\chi_G(-1) =
- \frac{(-1)(-2)(-3)}{2} = 3$, which is the number of regions of
$(\P(G), \HH(G))$ as stated in Theorem~\ref{thm:mixedgraph}. 

As we know from the proof of Theorem~\ref{thm:mixedgraph}, we can also
use Ehrhart reciprocity to interpret the evaluation of $\chi_G$ at other negative integers. Specifically, $- \chi_G(-t)$  is the number of lattice points in
$(t-1)\P$ with multiplicities, where the multiplicity of each point is
the number of closed regions of the inside-out polytope that contain
it. For instance, $- \chi_G(-2) = (-1)\frac{(-2)(-3)(-4)}{2} = 12$ and
$- \chi_G(-3) = (-1) \frac{(-3)(-4)(-4)}{2} = 30$. In
Table~\ref{tab:triangle}, we list the lattice points in $\P$ and $2\P$
and the regions that contain each point. Indeed, the multiplicities add
up to 12 and 30, respectively. 

Note that in this example, $\HH(G)$ is not transverse to $\P(G)$. The
two hyperplanes $x_1=x_3, x_2=x_3$ in $\HH(G)$ meet in a flat given by
$x_1=x_2=x_3$. This flat intersects only the boundary of $\P$ since the
equation $x_1=x_2$ is a supporting hyperplane of~$\P$. 
\end{example}

\begin{table}
\begin{tabular}{|c|c|c|} \hline
point $\z \in \P \cap \Z^3$ & regions containing $\z$ & multiplicity
of $\z$ \\ \hline
$(0,0,0)$ & $R_{123}, R_{132}, R_{312}$ & 3 \\
$(0,0,1)$ & $R_{123}$ & 1 \\
$(0,1,0)$ & $R_{132}, R_{312}$ & 2 \\
$(0,1,1)$ & $R_{123}, R_{132}$ & 2 \\
$(1,1,0)$ & $R_{312}$ & 1 \\
$(1,1,1)$ & $R_{123}, R_{132}, R_{312}$ & 3 \\ \hline
\end{tabular}
\vskip0.5cm

\begin{tabular}{|c|c|c|} \hline
point $\z \in 2\P \cap \Z^3$ & regions containing $\z$ & multiplicity
of $\z$ \\ \hline
$(0,0,0)$ & $R_{123}, R_{132}, R_{312}$ & 3 \\
$(0,0,1)$ & $R_{123}$ & 1 \\
$(0,0,2)$ & $R_{123}$ & 1 \\
$(0,1,0)$ & $R_{132}, R_{312}$ & 2 \\
$(0,1,1)$ & $R_{123}, R_{132}$ & 2 \\
$(0,1,2)$ & $R_{123}$ & 1 \\
$(0,2,0)$ & $R_{132}, R_{312}$ & 2 \\
$(0,2,1)$ & $R_{132}$ & 1 \\
$(0,2,2)$ & $R_{132}, R_{132}$ & 2 \\
$(1,1,0)$ & $R_{312}$ & 1 \\
$(1,1,1)$ & $R_{123}, R_{132}, R_{312}$ & 3 \\ 
$(1,1,2)$ & $R_{123}$ & 1 \\ 
$(1,2,0)$ & $R_{312}$ & 1 \\ 
$(1,2,1)$ & $R_{132}, R_{312}$ & 2 \\ 
$(1,2,2)$ & $R_{123}, R_{132}$ & 2 \\ 
$(2,2,0)$ & $R_{312}$ & 1 \\ 
$(2,2,1)$ & $R_{312}$ & 1 \\ 
$(2,2,2)$ & $R_{123}, R_{132}, R_{312}$ & 3 \\ 
\hline
\end{tabular}
\caption{Lattice points in $\P$ and $2\P$ and the regions that contain each point.}\label{tab:triangle}
\end{table}


\section{Open Problems} \label{open}

The search for optimal Golomb rulers remains a challenging open problem. This problem is related to our paper: the smallest positive integer that is not a root of $g_m(t)$ gives the length of an optimal Golomb ruler with $m+1$ markings.
Our methods present new open problems, from the computation of $g_m(t)$ to specific questions, e.g., about the period of $g_m(t)$ and the number of combinatorially different Golomb rulers with a given number of markings (this sequence starts with 1, 2, 10, 114, 2608, and 107498).
Our reciprocity theorem for chromatic polynomials of mixed graphs suggest open problems about these polynomials, from classification problems (which polynomials are chromatic polynomials of mixed graphs?) to interpretations of and relations between its coefficients.


\bibliographystyle{plain}

\setlength{\parskip}{0cm} 

\end{document}